\documentclass[reqno, 11pt, twoside]{amsart}

\usepackage[margin=3.5cm]{geometry}

\usepackage{amsmath,amssymb,amsthm,bm,colonequals,graphicx,mathrsfs,mathtools,microtype,stmaryrd, amsfonts, dsfont, comment, multicol}
\usepackage[shortlabels]{enumitem}
\usepackage[colorinlistoftodos]{todonotes}

\numberwithin{equation}{section}

\theoremstyle{plain}

\newtheorem{theorem}{Theorem}[section] 

\newtheorem{lemma}[theorem]{Lemma} 

\newtheorem{proposition}[theorem]{Proposition} 

\newtheorem{proposition-definition}[theorem]{Proposition-Definition}

\theoremstyle{definition}

\theoremstyle{remark}

\newcommand{\tr}{\operatorname{tr}}

\newcommand{\abs}[1]{\lvert #1 \rvert} 

\newcommand{\norm}[1]{\lvert #1 \rvert}

\newcommand{\belongs}{\subseteq}

\newcommand{\eps}{\epsilon}
\newcommand{\EE}{\mathbb{E}}
\newcommand{\GG}{\mathbb{G}}
\newcommand{\defeq}{\colonequals}

\newcommand{\maps}{\colon}

\newcommand{\inmat}[1]{\left[\begin{smallmatrix} #1 \end{smallmatrix}\right]}

\newcommand{\NN}{\mathbb{N}}
\newcommand{\ZZ}{\mathbb{Z}}
\newcommand{\QQ}{\mathbb{Q}}
\newcommand{\CC}{\mathbb{C}}
\newcommand{\FF}{\mathbb{F}}
\newcommand{\A}{\mathbb{A}}
\newcommand{\PP}{\mathbb{P}}

\DeclareMathOperator{\Gal}{Gal}

\DeclareMathOperator{\GL}{GL}
\DeclareMathOperator{\SL}{SL}

\def\RR{\mathbb{R}}

\newcommand{\ol}{\overline}

\renewcommand{\hat}{\widehat}

\renewcommand{\leq}{\leqslant}
\renewcommand{\le}{\leqslant}
\renewcommand{\geq}{\geqslant}
\renewcommand{\ge}{\geqslant}
\newcommand{\ve}{\varepsilon}
\renewcommand{\eps}{\varepsilon}

\usepackage[pagebackref=true]{hyperref}
\usepackage[alphabetic,initials]{amsrefs}

\title{Pairs of commuting integer matrices}

\author{Tim Browning}
\author{Will Sawin}
\author{Victor Y. Wang}

\address{IST Austria\\
Am Campus 1\\
3400 Klosterneuburg\\
Austria}
\email{tdb@ist.ac.at}

\address{Princeton University\\
Fine Hall\\ 304 Washington Rd\\
Princeton NJ 08540\\ USA}\email{wsawin@math.princeton.edu}

\address{Institute of Mathematics\\
Academia Sinica\\
Taipei 106319\\
Taiwan}
\email{vywang@as.edu.tw}


 \date{\today}

\begin{document}

\begin{abstract}
We prove  upper and lower bounds on the number of pairs of commuting $n\times n$ matrices
with integer entries in $[-T,T]$, as $T\to \infty$. 
Our work uses  Fourier analysis and leads to an analysis of exponential sums involving matrices over finite fields. These are bounded by combining a stratification result of Fouvry and Katz with a new result about the
 flatness of the commutator Lie bracket.
\end{abstract}

\maketitle

\thispagestyle{empty}
 \setcounter{tocdepth}{1}
 \tableofcontents

\section{Introduction}

 Let $\mathsf{M}_n$ denote the scheme of $n\times n$ matrices, for an integer
$n\ge 2$.
Let $\mathsf{V}_n\subset \mathsf{M}_n$ denote the closed subscheme $\tr = 0$.
Concretely, if $R$ is a ring, then
$\mathsf{M}_n(R)$ denotes the set of $n\times n$ matrices $A$ with entries in $R$,
and $\mathsf{V}_n(R)$ denotes the set of such matrices for which $\tr(A)=0$.
The aim of this note is to  count integral points of height $\le T$, as $T\to \infty$,
on the \emph{commuting variety} 
$$
\mathfrak{C}_n=\{(X,Y)\in \mathsf{M}_n(\CC)^2: XY=YX\},
$$
viewed as an affine variety in  $\mathsf{M}_n(\CC)^2 \cong \CC^{2n^2}$.
Thus we are interested in the behaviour of the counting function
$$
N(T) \defeq
\#\{(X,Y)\in \mathsf{M}_n(\ZZ)^2: \norm{X},\norm{Y}\le T,\; XY=YX\},
$$
as $T\to\infty$,
where $\norm{X}$ denotes the maximum modulus of any of the entries in a matrix $X\in \mathsf{M}_n(\ZZ)$.
It was shown by 
  Motzkin and Taussky in the paragraph after \cite[Thm.~6]{mot}
   that $\mathfrak{C}_n$  is irreducible and it follows from 
   work of Basili \cite[Thm.~1.2]{bas} that $\dim \mathfrak{C}_n=n^2+n$.
   Furthermore,   $\mathfrak{C}_n$ cannot be linear, since 
   it contains $(X,0)$ and $(0,Y)$,
   for any  $X,Y\in \mathsf{M}_n(\CC)$, 
    but it does not include the  point $(X,Y)$ if $X$ and $Y$ are non-commuting.  
In fact, a recursive formula for the
degree of $ \mathfrak{C}_n$ has been provided by
Knutson and Zinn-Justin \cite{knut}.

Since $\mathfrak{C}_n$ is defined by a system of homogeneous quadratic equations with coefficients in $\QQ$, we can view 
$\mathfrak{C}_n$ as an affine cone over a non-linear irreducible variety in $\PP^{2n^2-1}$ of dimension $n^2+n-1$, which is definable over $\QQ$. Appealing to the resolution of the {\em dimension growth conjecture} by Salberger \cite{salb}, it therefore follows that 
\begin{equation}
\label{DGC-bound}
N(T) 
\ll_{n,\eps} T^{n^2+n-1+\eps}
\end{equation}
for any $\eps>0$, where the implied constant depends at most on $n$ and $\ve$.
Moreover, the  $\eps$ can be removed if the variety has degree $\ge 5$ by \cite{DGCnew}, 
which is likely to be the case for $n\geq 3$ by the work in 
\cite{knut}.
On the other hand, we have the lower bound
\begin{equation}\label{eq:lower}
N(T)
\ge \#\{(X,Y)\in \mathsf{M}_n(\ZZ)\times \ZZ I_n: \norm{X},\norm{Y}\le T\}
\gg T^{n^2+1}
\end{equation}
coming from the special subvariety of scalar matrices $Y$,
where $I_n$ denotes the $n\times n$ identity matrix.
The following is our main result. 

\begin{theorem}
\label{THM:main1}
Let $T\ge 1$.
Then
$$N(T) \ll_{n} T^{n^2+2-
\frac{2}{n + 1}}.
$$
\end{theorem}

Note that 
Theorem \ref{THM:main1} improves on 
\eqref{DGC-bound} for $n\ge 3$. In fact our exponent is always within $1$ of that occurring in the lower bound \eqref{eq:lower}, which we expect to be the truth. 
Our approach uses
matrix identities and
ideas from harmonic analysis, which will lead us to 
 analyse the fibres of the commutator map
\begin{equation}
\label{scheme-bracket}
[\cdot,\cdot]\maps \mathsf{M}_n^2\to \mathsf{V}_n,\quad (X,Y)\mapsto XY-YX,
\end{equation}
where we recall that
 $\mathsf{V}_n\subset \mathsf{M}_n$ denotes the closed subscheme $\tr = 0$.
In this way we 
have been led to prove the following result, 
the analogue of which  was proved by 
 Larsen and Lu \cite{LarsenLu}
 in the multiplicative setting.

\begin{theorem}\label{THM:flat}
The map 
\begin{equation}
\label{complex-commutator}
[\cdot,\cdot](\CC)\maps \mathsf{M}_n(\CC)^2\to \mathsf{V}_n(\CC),\quad (X,Y)\mapsto XY-YX
\end{equation}
is flat over the open set $\mathsf{V}_n(\CC) \setminus \{0\}$.
\end{theorem}

As explained below in Proposition~\ref{dim-flat-equiv},
Theorem~\ref{THM:flat} is equivalent to the \emph{pointwise} bound
\begin{equation}
\label{complex-fibre-bound}
\dim \left\{(U,V)\in \mathsf{M}_n(\CC)^2: UV-VU = M\right\} \le n^2+1
\end{equation}
on the dimensions of the fibres of the commutator map
over matrices $M\in \mathsf{V}_n(\CC) \setminus \{0\}$.\footnote{In fact,
equality must hold in \eqref{complex-fibre-bound}, since  the map \eqref{complex-commutator} is surjective by work of Shoda \cite{shoda}.}
For comparison, earlier work has examined the dimensions of the fibres 
\emph{on average} over certain small families of matrices
in $\mathsf{V}_n(\CC) \setminus \{0\}$,
such as diagonal matrices (Knutson \cite{Knutson})
or rank-one matrices (Neubauer \cite{Neubauer}).
Moreover, Young \cite{Young}
proved a finer conjecture of Knutson \cite{Knutson}
concerning the irreducible components of the diagonal commutator scheme.
Our new methods may be capable of re-proving,
and generalising to other families,
some of these earlier results.

\subsection*{Acknowledgements}
The authors are very grateful to Alina Ostafe, Matthew Satriano and  Igor Shparlinski
for drawing their attention to this problem and for useful comments, and to  Michael Larsen and Peter Sarnak for their helpful correspondence. 
We also thank the referee for their valuable input.  
While working on this paper
the first author was supported  by 
a FWF grant (DOI 10.55776/P36278),
the second author by a Sloan Research Fellowship, 
and 
the third author by the European Union's Horizon 2020 research and 
innovation programme under the Marie Sk\l{}odowska-Curie Grant Agreement No.~101034413.

\section{Background on flatness}

Before proceeding, we recall some classical criteria for flatness, which are useful both for interpreting Theorem \ref{THM:flat} and for proving it.
Throughout this paper, a \emph{variety} will be a locally closed subset of a projective space $\PP^N$ over a field $k$ (which is often called a \emph{quasi-projective variety}).
If $k$ is algebraically closed, then varieties $V$ over $k$ may be completely understood in terms of their $k$-points $V(k)$ by
\cite[Chapter~II, Proposition~2.6]{hart},
and
we will take this classical viewpoint 
 whenever possible. 

\begin{proposition}
\label{dim-flat-equiv}
Let $f\maps E_1\to E_2$ be a morphism of smooth, irreducible varieties over a field $k$.
Let $k'$ be a field containing $k$.
Then the following are equivalent:
\begin{enumerate}
\item The morphism $f$ is flat.

\item The morphism $f\otimes k'\maps E_1\otimes k'\to E_2\otimes k'$
is flat.

\item For every point $y\in E_2(\ol{k})$,
the fibre $f^{-1}(y)$ is either empty or of dimension exactly $\dim{E_1}-\dim{E_2}$.

\item For every point $y\in E_2(\ol{k})$,
we have $\dim(f^{-1}(y))\le \dim{E_1}-\dim{E_2}$.
\end{enumerate}
\end{proposition}

\begin{proof}
(1)$\Leftrightarrow$(2):
Flatness is preserved by base change from $k$ to $k'$.
(See \cite[\href{https://stacks.math.columbia.edu/tag/047C}{Tag~047C}]{stacks-project}
or \cite[\href{https://stacks.math.columbia.edu/tag/02JZ}{Tag~02JZ}]{stacks-project}.)
The key point is that the inclusion homomorphism $k\to k'$ is flat.



(1)$\Rightarrow$(3):
This follows from \cite[Theorem~15.1 (or paragraph~1 of \S~23)]{mats}.

(3)$\Rightarrow$(1):
This follows from miracle flatness \cite[Corollary to Theorem~23.1]{mats},
which applies since $E_1$ and $E_2$ are smooth, irreducible varieties over $k$.

(3)$\Rightarrow$(4):
This is trivial.

(4)$\Rightarrow$(3):
Any non-empty fibre $f^{-1}(y)$
has dimension at least $\dim{E_1}-\dim{E_2}$,
and thus exactly $\dim{E_1}-\dim{E_2}$ if (4) holds.
\end{proof}

While a direct proof of Theorem \ref{THM:flat} is possible 
by induction on $n$ (via  algebraic geometry, Grassmannians,
symmetry, and dimension counting), we shall find it more convenient to  prove instead a characteristic $p$ version.
We will  prove the following  result 
in \S~\ref{prove-flatness} 
by induction on $n$
for $p\ge n$, using  symmetry and point counting.

\begin{theorem}\label{THM:dimension}
For all primes $p\gg_n1$,
and for all points $M\in \mathsf{V}_n(\ol{\FF}_p) \setminus \{0\}$, we have
$$
\dim \left\{(U,V)\in \mathsf{M}_n(\overline{\FF}_p)^2: UV-VU = M\right\} \le n^2+1.
$$
\end{theorem}

We shall prove that this result is equivalent to Theorem \ref{THM:flat}.
In fact Theorem \ref{THM:flat} and this equivalence are not needed for the Diophantine results in this paper, but are included purely for the relevance of Theorem \ref{THM:flat} to algebraic geometry.
The key point is that miracle flatness, such as the equivalence  (1) $\Leftrightarrow$ (3) in Proposition 
\ref{dim-flat-equiv},  converts flatness statements like  Theorem~\ref{THM:flat} 
into dimension statements like Theorem~\ref{THM:dimension}.
To deal with the fact that one statement occurs in characteristic zero and one statement occurs in characteristic $p$, we use standard ``spreading out'' arguments\footnote{A general reference for such arguments is \cite{Poonen}*{Theorem~3.2.1(iv)}, but we don't  really need this particular form of the theorem.} to convert between characteristics $0$ and $p$.
This consists of setting everything up over $\operatorname{Spec} \mathbb Z$ and applying general results that natural properties (in this case, flatness) define open subsets of varieties over $\operatorname{Spec} \mathbb Z$.
At the heart of this lies the observation that an open subset of $\operatorname{Spec} \mathbb Z$ contains the characteristic $0$ point $\operatorname{Spec} \mathbb Q$ if and only if it contains the characteristic $p$ point $\operatorname{Spec} \mathbb F_p$, for all but finitely many primes $p$.

\begin{proof}[Proof of the equivalence of Theorems \ref{THM:flat} and  \ref{THM:dimension}]
We  broadly follow \cite{LarsenLu}*{\S~5}.
Let $\mathsf{V}_n-0$ be the open subscheme of $\mathsf{V}_n$
with zero section removed, over $\ZZ$.
Thus $(\mathsf{V}_n-0)(k) = \mathsf{V}_n(k)\setminus \{0\}$, for every field $k$.
Let $\mathsf{C}_n \subseteq \mathsf{M}_n^2$ be the commuting scheme
$XY=YX$, over $\ZZ$.
Let $f\maps \mathsf{M}_n^2-\mathsf{C}_n\to \mathsf{V}_n-0$
be the commutator map in  \eqref{scheme-bracket},
restricted to $\mathsf{M}_n^2-\mathsf{C}_n$.
Let $O\subseteq \mathsf{M}_n^2-\mathsf{C}_n$ be the \emph{flat locus} of $f$;
i.e.~the set of points $x\in \mathsf{M}_n^2-\mathsf{C}_n$ at which the morphism $f$ is flat.
Then $O$ is open by
\cite[\href{https://stacks.math.columbia.edu/tag/0399}{Tag~0399}]{stacks-project}.
Theorem~\ref{THM:flat} is the statement that the base change $f\otimes \CC \colon (\mathsf{M}_n^2-\mathsf{C}_n)_\mathbb C \to (\mathsf{V}_n -0)_{\mathbb C}$ is flat.
By the equivalence (1)$\Leftrightarrow$(2) in Proposition~\ref{dim-flat-equiv}, this is equivalent to the flatness of $f\otimes \QQ \colon (\mathsf{M}_n^2-\mathsf{C}_n)_\mathbb Q \to (\mathsf{V}_n -0)_{\mathbb Q}$,
which is equivalent to $O$ containing the entire characteristic $0$ fibre, $(\mathsf{M}_n^2-\mathsf{C}_n)_\mathbb Q$.
An open set of a scheme of finite type over $\mathbb Z$ contains the entire characteristic $0$ fibre if and only if it contains all but finitely many characteristic $p$ fibres,\footnote{The analytically minded reader may think of this as a consequence of the Lang--Weil estimate for the complement of the open set in question.}
so Theorem~\ref{THM:flat} is equivalent to $O$ containing $(\mathsf{M}_n^2-\mathsf{C}_n)_{\mathbb F_p}$ for all but finitely many $p$. We 
call this equivalence $(\star)$.

If $O$ contains $(\mathsf{M}_n^2-\mathsf{C}_n)_{\mathbb F_p}$,
then since the base change of a flat morphism is flat (by part (1) of
\cite[\href{https://stacks.math.columbia.edu/tag/047C}{Tag~047C}]{stacks-project}),
the base change of $f$ along $(\mathsf{V}_n -0)_{ \overline{\mathbb F}_p} \to (\mathsf{V}_n -0)_{ \mathbb Z}$ is flat.
This base change is $f\otimes\ol{\FF}_p \colon (\mathsf{M}_n^2-\mathsf{C}_n)_{\overline{\mathbb F}_p} \to (\mathsf{V}_n -0)_{\overline{\mathbb F}_p}$.
From this (and the implication (1)$\Rightarrow$(4) in Proposition~\ref{dim-flat-equiv})
the statement of Theorem~\ref{THM:dimension} follows, since
$$\dim \mathsf{M}_n(\ol{\FF}_p)^2 - \dim \mathsf{V}_n(\ol{\FF}_p) = 2n^2-(n^2-1)=n^2+1.$$
Thus Theorem~\ref{THM:flat} implies Theorem \ref{THM:dimension}.

For the reverse implication, we appeal to  miracle flatness, which for the generality we need is established in  \cite[\href{https://stacks.math.columbia.edu/tag/00R4}{Tag~00R4}]{stacks-project} or \cite[Theorem~23.1]{mats}. This states that a morphism between regular Noetherian schemes is flat in a neighbourhood of any point where the dimension of the fibre plus the dimension of the target equals the dimension of the source. But then Theorem~\ref{THM:dimension} implies that
for all $p\gg_n 1$, the set $O$ contains the entire characteristic $p$ fibre of $\mathsf{M}_n^2 - \mathsf{C}_n$, which by the equivalence $(\star)$ described above implies Theorem~\ref{THM:flat}.
\end{proof}

\section{Matrix exponential sums}

Over any field $k$, 
we will often make the identification $\mathsf{M}_n(k)\cong k^{n^2}$. 
The usual dot product on $X,Y\in \mathsf{M}_n(k)$
is then given by $\tr(X^tY) = \sum_i \sum_j X_{ji}Y_{ji}$.
For  the purposes of harmonic analysis to come, however, 
it will be more convenient to work with
the pairing $\mathsf{M}_n(k)\times \mathsf{M}_n(k)\to k$
given by $(X,Y)\mapsto \tr(XY) = \sum_i \sum_j X_{ij}Y_{ji}$,
which can be viewed as the usual dot product after an invertible linear change of variables $X\mapsto X^t$.
(Abstractly, the key point
is that this pairing is \emph{non-degenerate},
in the sense that
for any non-zero $X\in \mathsf{M}_n(k)$,
the linear form $Y\mapsto \tr(XY)$ is non-zero.)

Given 
 $A,B\in \mathsf{M}_n(\ZZ)$, our work will lead us to analyse the complete exponential sum
\begin{equation}
\label{eq:Sp}
S(A,B;p)
\defeq \sum_{\substack{(U,V)\in \mathsf{M}_n(\FF_p)^2\\
UV-VU=0
}} e_p(\tr(AU+BV)),
\end{equation}
for a sufficiently large prime $p$.  
It follows from the point count of 
Feit--Fine \cite{FeitFine} that 
\begin{equation}\label{eq:FF}
S(0,0;p) \ll_n p^{n^2+n}.
\end{equation}
We seek to find conditions on $A,B$ under which we can show  that cancellation occurs in
$S(A,B;p)$.

We begin by recording the observation that 
\begin{equation}
\label{define-key-local-fourier-transforms}
S(A,B;p)
= \frac{1}{\#\mathsf{M}_n(\FF_p)} \sum_{U,V,Z\in \mathsf{M}_n(\FF_p)} e_p(\tr(Z(UV-VU) + AU+BV)).
\end{equation}
The cyclic property of the trace ensures that 
$\tr(AB)=\tr(BA)$ and 
$\tr(ABC)=\tr(CAB)=\tr(BCA)$, for any $A,B,C\in \mathsf{M}_n(\FF_p)$. Hence
$\tr(Z(UV-VU) + AU)=
\tr(U(VZ-ZV + A))$, so that we can   average over $U$  
to conclude that 
\begin{equation*}
S(A,B;p)
= \sum_{\substack{V,Z\in \mathsf{M}_n(\FF_p)\\
VZ-ZV+A = 0}}
e_p(\tr(BV)).
\end{equation*}
In particular, it is clear that 
\begin{equation}\label{eq:LW}
|S(A,B;p)|\leq S(A,0;p)=
\#\left\{(V,Z)\in \mathsf{M}_n(\FF_p)^2: VZ-ZV+A = 0\right\},
\end{equation}
for any $B\in \mathsf{M}_n(\ZZ)$. 
We are now ready to prove the following result.

\begin{lemma}\label{lem:exp}
If $p\nmid A$ or $p\nmid B$, then 
$$
S(A,B;p) \ll_n p^{n^2+1} \bm{1}_{p\mid \tr(A) } \bm{1}_{p\mid \tr(B) }.
$$
\end{lemma}

\begin{proof}
We may assume that $p$ is sufficiently large in terms of $n$, else  the result is trivial.
If $p\nmid A$,
then it follows from applying Theorem~\ref{THM:dimension}
and the Lang--Weil bound
in 
\eqref{eq:LW}
that 
\begin{align*}
S(A,B;p)
\ll_n p^{n^2+1} \bm{1}_{p\mid \tr(A)}.
\end{align*}
Similarly, if $p\nmid B$, then averaging over $V$ in \eqref{define-key-local-fourier-transforms} gives
\begin{equation*}
S(A,B;p)
\ll \#\left\{(U,Z)\in \mathsf{M}_n(\FF_p)^2: ZU-UZ+B = 0\right\}
\ll_n p^{n^2+1} \bm{1}_{p\mid \tr(B)}.
\end{equation*}
Combining these bounds,  we arrive at the statement of the lemma. 
\end{proof}

\section{Fourier analysis and proof of the main result}

In this section we 
prove Theorem \ref{THM:main1} via   harmonic analysis, by combining 
work of  Fouvry and Katz \cite{FK} with  
 our discussion  in the previous section on exponential sums over finite fields.
We shall assume throughout this section that $n\geq 2$.

Let $p$ be an auxiliary prime in the interval $[T,2T^2]$
whose size will be optimised later.
Then 
$$
N(T)\leq 
\#\{(X,Y)\in \mathsf{M}_n(\ZZ)^2: \norm{X},\norm{Y}\le T,\; p\mid XY-YX\}.
$$
Consider the function
$$
w(X)\defeq \prod_{1\leq i,j\leq n}\left(\frac{\sin \pi X_{ij}}{\pi X_{ij}}\right)^2,
$$
for any $X\in \mathsf{M}_n(\RR)$. Identifying $\mathsf{M}_n(\RR)$ with $\RR^{n^2}$, this  has Fourier transform 
\begin{equation}\label{fourier}
\hat{w}(A)=\int_{\mathsf{M}_n(\RR)} w(X)e(\tr (AX)) \mathrm{d} X =
\prod_{1\leq i,j\leq n}\max\{1-|A_{ij}|,0\},
\end{equation}
where $\mathrm{d} X=\prod_{i,j}\mathrm{d} X_{ij}$. 
Moreover it is clear that $(\frac{4}{\pi^2})^{n^2}\leq w(\frac{1}{2}X)\leq 1$ if $|X|\leq 1$.
It now follows that 
\begin{align*}
N(T) 
&\ll_n  \sum_{\substack{(X,Y) \in \mathsf{M}_n(\ZZ)^2\\
p \mid XY-YX}} w\Big(\frac{1}{2T}X\Big)w\Big(\frac{1}{2T}Y\Big)\\
&=
\sum_{\substack{(U,V) \in \mathsf{M}_n(\FF_p)^2\\
UV-VU=0}}   
\sum_{\substack{X' \in \mathsf{M}_n(\ZZ)}} w\Big(\frac{U+pX'}{2T}\Big)
\sum_{\substack{Y' \in \mathsf{M}_n(\ZZ)}} 
w\Big(\frac{V+pY'}{2T}\Big).
\end{align*}
An application of Poisson summation 
reveals that 
\begin{align*}
\sum_{\substack{X' \in \mathsf{M}_n(\ZZ)}} w\Big(\frac{U+pX'}{2T}\Big)
&=
\sum_{A\in \mathsf{M}_n(\ZZ)} 
\int_{\mathsf{M}_n(\RR)} 
w\Big(\frac{U+pX'}{2T}\Big)
e(\tr (AX')) \mathrm{d} X'\\
&=
\left(\frac{2T}{p}\right)^{n^2}
\sum_{A\in \mathsf{M}_n(\ZZ)} 
e_p(-\tr (AU))  \hat w\left(\frac{2T}{p}A\right),
\end{align*}
on making an obvious change of variables. We have a similar identity for the sum over $Y'$, all of which leads to the expression
\begin{align*}
N(T) 
&\ll_n  
\left(\frac{T}{p}\right)^{2n^2}
\sum_{(A,B)\in \mathsf{M}_n(\ZZ)^2}  
\hat w\left(\frac{2T}{p}A\right)
\hat w\left(\frac{2T}{p}B\right)
S(A,B;p),
\end{align*}
in the notation of \eqref{eq:Sp}.

Appealing to \eqref{fourier}, we finally deduce that 
\begin{align*}
N(T) 
&\ll_n  
\left(\frac{T}{p}\right)^{2n^2}
\sum_{\substack{(A,B)\in \mathsf{M}_n(\ZZ)^2\\
|A|,|B|\le \frac12p/T}} 
|S(A,B;p)|.
\end{align*}
If $p\mid (A,B)$ then only  $A=B=0$ contribute to  the sum. Thus  
it follows from 
\eqref{eq:FF} that 
\begin{equation}\label{eq:N-U}
N(T) 
\ll_n  
\left(\frac{T}{p}\right)^{2n^2}
\left(p^{n^2+n} + U(p,T)
 \right),
\end{equation}
where
$$
 U(p,T)=
\sum_{\substack{(A,B)\in \mathsf{M}_n(\ZZ)^2\\
|A|,|B|\le \frac12p/T\\
p\nmid (A,B)
}} 
|S(A,B;p)|.
$$
We now appeal to a {\em stratification} result  by Fouvry and Katz \cite[Thm.~1.1]{FK}
to estimate $U(p,T)$.
This produces 
subschemes $V_{2n^2}\subset V_{2n^2-1}\subset \cdots \subset V_2\subset V_1\subset \A_\ZZ^{2n^2}$, with $\dim(V_j\otimes \CC)\leq 2n^2-j$ for $1\leq j\leq 2n^2$, 
such that 
$$
S(A,B;p)\ll_n p^{\frac{n^2+n}{2}+\frac{j-1}{2}}, 
$$
whenever the reduction modulo $p$ of $(A,B)$ corresponds to vector in $\A^{2n^2}(\FF_p)\setminus V_j(\FF_p)$, under the isomorphism $\mathsf{M}_n^2\cong \A^{2n^2}$.
It is convenient to put $V_0=\A_\ZZ^{2n^2}$ and $V_{2n^2+1}=\emptyset$.
Combining 
this with Lemma 
\ref{lem:exp}, we therefore deduce that 
\begin{align*}
U(p,T)
&\ll_n  
\sum_{j=1}^{2n^2+1}
\sum_{\substack{(A,B)\in R_j}}
\min\left\{p^{\frac{n^2+n}{2}+\frac{j-1}{2}}~,~  p^{n^2+1}\right\},
\end{align*}
where
$R_j$ denotes the set of 
$(A,B)\in \mathsf{M}_n(\ZZ)^2$ with 
$|A|,|B|\le \frac12p/T$ for which 
the reduction of $(A,B) \bmod{p}$ is non-zero and 
belongs to the set $V_{j-1}(\FF_p)\setminus V_j(\FF_p)$.
We have $R_j\subseteq R_j'$, where 
$R_j'$ is the set of 
$(A,B)\in \mathsf{M}_n(\ZZ)^2$ with 
$|A|,|B|\le \frac12p/T$ for which 
the reduction of $(A,B) \bmod{p}$ belongs to  $V_{j-1}(\FF_p)$.
We now recall the statement of 
 \cite[Lemma~4]{bhb} with $r=1$. Given a subscheme $W\subset \A_\ZZ^N$
 and an affine variety $V\subset \A_{\FF_p}^N$, with 
 $\dim(W\otimes \CC)\leq \ell$ and $\dim V\leq k$, this states that 
 $$
 \#\left\{
 \mathbf{t}\in W\cap \ZZ^N : |\mathbf{t}|\leq H, ~\mathbf{t} \bmod{p}\in V(\FF_p) 
 \right\}\ll_{D, N} H^\ell p^{k-\ell} +H^k,
 $$
 for any $H\geq 1$, where $D=\max\{\deg(W\otimes \CC),\deg V\}$. 
 Note that 
$H^\ell p^{k-\ell}\leq H^k$ if $k\leq \ell$ and $H\leq p$.
  We apply this with 
 $W=\mathsf{M}_n^2$,
 $V=V_{j-1}$,
 $\ell=2n^2$,
 $k=2n^2-j+1$,
 and $H=p/T$.
 Noting that $p\geq T$, it now follows that 
$$
\#R_j\leq 
\#R_j'\ll_n \left(\frac{p}{T}\right)^{2n^2-j+1}.
$$
But then we may conclude that 
\begin{align*}
U(p,T)
&\ll_n   \left(\frac{p}{T}\right)^{2n^2+1} \left(
\sum_{j=1}^{n^2-n+3}
\left(\frac{T}{\sqrt{p}}\right)^{j}
p^{\frac{n^2+n-1}{2}}
+
\sum_{j=n^2-n+4}^{2n^2+1}
\left(\frac{T}{p}\right)^{j}
p^{n^2+1}\right).
\end{align*}
Since $p\in [T,2T^2]$ we see that the first term is maximised at the largest value of $j$, while the second term is maximised when $j$ is least. 
This leads to the expression
\begin{align*}
U(p,T)
&\ll_n   \left(\frac{p}{T}\right)^{2n^2+1} \left(T^{n^2-n+3} p^{n-2} +
T^{n^2-n+4}p^{n-3}\right)\\
&\ll_n   \left(\frac{p}{T}\right)^{2n^2}\cdot  T^{n^2-n+2} p^{n-1},
\end{align*}
since $p\geq T$.

Returning to \eqref{eq:N-U}, we have now established that 
$$
N(T) 
\ll_n  
\frac{T^{2n^2}}{p^{n^2-n}}+
T^{n^2-n+2} p^{n-1}.
 $$
This is optimised at $p^{n^2-1} \asymp T^{n^2+n-2}$, which leads to the  final bound for $N(T)$ recorded in 
 Theorem \ref{THM:main1}.

The idea of
choosing an auxiliary modulus $p$
and using harmonic analysis 
has been applied to other varieties in the past, such as by
 Fujiwara \cites{Fujiwara1,Fujiwara2}, and by 
Shparlinski and Skorobogatov \cites{SS,Skoro}.
Moreover, using finer geometric information,
Heath-Brown \cite{Heath}
has given stronger results in many cases
by working with a composite modulus $pq$.
Once combined with Fouvry--Katz stratification \cite{FK}, 
the simplest version of Fujiwara's method gives the following
general result.

\begin{theorem}
Let $V\subset \A_\ZZ^N$ be a subscheme with $\dim(V\otimes\CC)=D$.
Assume that 
there is a positive density set  $\mathcal{P}$  
of  primes such that for all $p\in \mathcal{P}$,  
and for all non-zero vectors $c\in \FF_p^N$,
we have
\begin{equation}
\label{general-condition}
\sum_{x\in V(\FF_p)} e_p(c_1x_1+\dots+c_Nx_N)
\ll_V p^{D-L},
\end{equation}
for some $L$ such that $2L\in \ZZ$.
Then for $T\ge 1$, we have
\begin{equation}
\label{general-bound}
\#\{x\in V(\ZZ)\cap [-T,T]^N\} \ll_V
T^{D-L+\frac{L^2}{N-D+L}}.
\end{equation}
\end{theorem}

This bound 
improves on
the \emph{dimension growth bound} $O_{\ve,V}(T^{D-1+\eps})$
of Salberger \cite{salb} and Vermeulen \cite{AffDGC}
(when it applies), provided only that 
$N-D\geq 4$ and 
$L\geq \frac{3}{2}$
in \eqref{general-condition}. 
Note that 
\eqref{general-bound} recovers the bound in Theorem 
\ref{THM:main1} when  $N = 2n^2$, $D = n^2+n$ and $L = n-1$.
Moreover, if $V$ lies in an algebraic family\footnote{That is, $V$ is the fibre over a $\ZZ$-point of a dominant morphism $\Delta\maps W\to M$, where $W$ is a subscheme of an affine space and $M$ is a closed subscheme of an affine space.
For example, if $V$ is a complete intersection, then one could simply ask for the sum of the degrees and number of variables of the defining equations to be bounded.}
 $F$
in the sense of Bonolis--Kowalski--Woo \cite{uniformFK},
then one could hope for \eqref{general-condition} and \eqref{general-bound} to be made fairly uniform over $V$,
perhaps with a polylogarithmic dependence on the coefficients
as in work of Bonolis--Pierce--Woo \cite{BPW}.

\section{Flatness of the commutator map}
\label{prove-flatness}

In this section we shall prove 
Theorem~\ref{THM:dimension}.
We shall often find it notationally convenient to adopt the expectation notation 
$\EE_{g\in S}=\frac{1}{\#S}\sum_{g\in S}$, for any non-empty subset $S\subseteq \GL_n(\FF_q)$.
Let $p$ be any prime.
(Eventually, we will assume $p\ge n$.)
Fix a non-zero $M\in \mathsf{V}_n(\ol{\FF}_p)$ and 
note that 
\begin{equation*}
\dim\{(U,V)\in \mathsf{M}_n(\ol{\FF}_p)^2: UV-VU = M\}
= -1 + \dim \mathcal{X},
\end{equation*}
where
\begin{equation*}
\mathcal{X} \defeq
\{(U,V,\lambda)\in \mathsf{M}_n(\ol{\FF}_p)^2 \times \GG_m(\ol{\FF}_p): UV-VU = \lambda M\}.
\end{equation*}
We shall bound $\dim \mathcal{X}$ via point counting on $\mathcal{X}$.
We take $q=p^m\to \infty$, with $m\in \NN$ sufficiently large so that
$M\in \mathsf{V}_n(\FF_q)$ and
the variety
$\mathcal{X}$
has an irreducible component that is
invariant under $\Gal(\ol{\FF}_p/\FF_q)$.
It will suffice to estimate 
the quantity
\begin{align*}
\Sigma(M)&=
\frac{1}{q}\sum_{U,V\in \mathsf{M}_n(\FF_q)}\sum_{\lambda\in \GG_m(\FF_q)}
\bm{1}_{UV-VU = \lambda M} ,
\end{align*}
since  then
\begin{equation}\label{eq:games}
\Sigma(M)
\geq (1+O_n(q^{-1/2}))\, q^{\dim\{(U,V)\in \mathsf{M}_n(\ol{\FF}_p)^2: UV-VU = M\}}
\end{equation}
by  the Lang--Weil estimate
for the variety $\mathcal{X}$.

Let $\psi(\cdot) \defeq e_p(\tr_{\FF_q/\FF_p}(\cdot))$ on $\FF_q$.
We proceed by noting that 
\begin{align*}
\Sigma(M)
&= \frac{1}{q\cdot \#\mathsf{M}_n(\FF_q)}\sum_{U,V,Z\in \mathsf{M}_n(\FF_q)}
\sum_{\lambda\in \GG_m(\FF_q)} \psi(\tr(Z(UV-VU-\lambda M))) \\
&= \frac{1}{\#\mathsf{M}_n(\FF_q)} \sum_{U,V,Z\in \mathsf{M}_n(\FF_q)}
\psi(\tr(Z(UV-VU)))\, (\bm{1}_{\tr(ZM)=0} - q^{-1}) \\
&= \sum_{\substack{V,Z\in \mathsf{M}_n(\FF_q) \\ VZ-ZV=0}} (\bm{1}_{\tr(ZM)=0} - q^{-1}),
\end{align*}
where in the last step we average over $U$ using the cyclic property of the trace function.
We now exploit the linearity of the equations $VZ-ZV=0$ and $\tr(ZM)=0$
in the coordinates of $Z$.
Since the hyperplane $\tr(ZM)=0$ has codimension either $0$ or $1$
in the linear space $C(V)\defeq \{Z\in \mathsf{M}_n(\ol{\FF}_q):VZ-ZV=0\}$, we find that
\begin{equation*}
\begin{split}
\Sigma(M)
&= \sum_{V\in \mathsf{M}_n(\FF_q)} q^{\dim C(V)}
(q^{-1}\bm{1}_{C(V)\not\belongs M^\perp}+\bm{1}_{C(V)\belongs M^\perp}-q^{-1}) \\
&= \sum_{V\in \mathsf{M}_n(\FF_q)} q^{\dim C(V)}
(1-q^{-1}) \bm{1}_{C(V)\belongs M^\perp},
\end{split}
\end{equation*}
where $M^\perp\defeq \{A\in \mathsf{M}_n(\ol{\FF}_q): \tr(AM)=0\}$.
Then, since $\dim C(V)$ depends only on the conjugacy class of $V$,
we find on averaging over conjugates of $V$ that
\begin{equation}
\label{G-averaging}
\Sigma(M)
= \sum_{V\in \mathsf{M}_n(\FF_q)} q^{\dim C(V)}
(1-q^{-1}) L(V,M),
\end{equation}
where
\begin{equation}
\label{define-LVM}
L(V,M)
\defeq \EE_{g\in \GL_n(\FF_q)} (\bm{1}_{C(gVg^{-1})\belongs M^\perp})
\le 1.
\end{equation}

The orbit of $V$ under the conjugation action of $\GL_n(\FF_q)$
has cardinality
\begin{equation*}
\mathscr{O}(V)
= \frac{\#\GL_n(\FF_q)}{\#(\GL_n(\FF_q)\cap C(V))},
\end{equation*}
by the orbit-stabiliser theorem.
Observe that  $\#\GL_n(\FF_q)\leq \#M_n(\FF_q)=q^{n^2}$. 
Moreover, the variety $\GL_n(\ol{\FF}_q)\cap C(V)$, which contains the identity matrix $I_n$, is a non-empty open subset of the linear space $C(V)$.
Thus $\#(\GL_n(\FF_q)\cap C(V))
\gg_n q^{\dim C(V)}$ by the Lang--Weil estimate, whence
$$
\mathscr{O}(V)
\ll_n q^{n^2 - \dim C(V)}.
$$
Therefore, if $K_n(\FF_q)\subseteq \mathsf{M}_n(\FF_q)$ denotes
a complete set of representatives for the conjugation action,
then breaking \eqref{G-averaging} into orbits gives
\begin{equation}
\begin{split}
\label{orbit-collapse}
\Sigma(M)
&= \sum_{V\in K_n(\FF_q)} \mathscr{O}(V)
q^{\dim C(V)} (1-q^{-1}) L(V,M) \\
&\ll_n q^{n^2} \sum_{V\in K_n(\FF_q)} L(V,M).
\end{split}
\end{equation}

Lemma~\ref{main-lem} will non-trivially bound $L(V,M)$.
The proof inducts on $n$
and uses an auxiliary averaging result, Lemma~\ref{aux-lem}.
First, we define some useful projection maps.
Given integers $1\le k\le n-1$ and a matrix $M\in \mathsf{M}_n(\FF_q)$,
define block matrices $p_{ij}(M)$ so that
\begin{equation*}
M = \inmat{p_{11}(M) & p_{12}(M) \\ p_{21}(M) & p_{22}(M)},
\end{equation*}
where $p_{ij}(M)$ has $k\bm{1}_{i=1} + (n-k)\bm{1}_{i=2}$ rows
and $k\bm{1}_{j=1} + (n-k)\bm{1}_{j=2}$ columns.

Lemma~\ref{aux-lem} concerns some averages over $\GL_n(\FF_q)$.
First we study averages over a certain subgroup.
Let $$E\defeq \{N\in \mathsf{M}_n(\FF_q): p_{11}(N) = 0,
\; p_{21}(N) = 0,
\; p_{22}(N) = 0\}
\cong \FF_q^{k\times (n-k)}.$$
Then $E$ is a vector space
such that $N_1N_2=0$ for all $N_1,N_2\in E$.
In particular,
$$1+E \defeq \{1+N: N\in E\}$$ is an abelian subgroup of $\GL_n(\FF_q)$,
where $(1+N)^{-1} = 1-N$ for all $N\in E$.

\begin{lemma}
\label{E-avg-lem}
Let $M\in \mathsf{M}_n(\FF_q)$,
with $M\ne 0$.
Let $1\le k\le n-1$.
Then the following hold:
\begin{enumerate}
\item We have
\begin{equation*}
\EE_{h\in 1+E} (\bm{1}_{p_{11}(h^{-1}Mh) = 0})
\le q^{-k}
+ \min_{h\in 1+E} (\bm{1}_{(p_{11},p_{21})(h^{-1}Mh) = 0}).
\end{equation*}

\item We have
\begin{equation*}
\EE_{h\in 1+E} (\bm{1}_{(p_{11},p_{22})(h^{-1}Mh) = 0})
\le q^{-(n-1)}
+ \min_{h\in 1+E} (\bm{1}_{(p_{11},p_{21},p_{22})(h^{-1}Mh) = 0}).
\end{equation*}

\item We have
\begin{equation*}
\EE_{h\in 1+E} (\bm{1}_{(p_{11},p_{12},p_{22})(h^{-1}Mh) = 0})
\le q^{-(n-1)}.
\end{equation*}
\end{enumerate}
\end{lemma}

\begin{proof}
First we record a general calculation:
for $N\in E$ we have
\begin{equation}
\begin{split}
\label{NMN}
&(1-N)M(1+N) \\
&= (1-N)(M+MN) \\
&= M+MN-NM-NMN \\
&= M + \inmat{0 & p_{11}(M)p_{12}(N) \\ 0 & p_{21}(M)p_{12}(N)}
- \inmat{p_{12}(N)p_{21}(M) & p_{12}(N)p_{22}(M) \\ 0 & 0}
- \inmat{0 & p_{12}(N)p_{21}(M)p_{12}(N) \\ 0 & 0}.
\end{split}
\end{equation}

(1):
Both sides of (1) are invariant under conjugation of $M$ by any element of $1+E$.
Moreover, if $p_{11}(h^{-1}Mh)$ is never zero,
then (1) is trivial.
Thus, after conjugating $M$ if necessary,
we may assume that $p_{11}(M)=0$.
Then, writing $h=1+N$, we get
\begin{align}
p_{11}(h^{-1}Mh)
&= -p_{12}(N)p_{21}(M),
\label{p11} \\
p_{21}(h^{-1}Mh)
&= p_{21}(M),
\label{p21}
\end{align}
by \eqref{NMN}.
In particular, by \eqref{p11},
\begin{equation*}
\EE_{h\in 1+E} (\bm{1}_{p_{11}(h^{-1}Mh) = 0})
= \EE_{N\in E} (\bm{1}_{p_{12}(N)p_{21}(M) = 0}).
\end{equation*}
Given a matrix $N\in E$, the condition $p_{12}(N)p_{21}(M) = 0$ holds
if and only if each \emph{row} of $p_{12}(N)$
lies in the left kernel $\mathscr{L}\belongs \FF_q^{n-k}$
of the $(n-k)\times k$ matrix $p_{21}(M)$.
If $p_{21}(M)\ne 0$,
then $\mathscr{L}$ is orthogonal to a non-zero column of $p_{21}(M)$.
Hence, one of the coordinates of $\mathscr{L}$
is a linear function of the others.
Thus, one of the columns of $p_{12}(N)$
is uniquely determined by others.
Since each column of $p_{12}(N)$ has $k$ entries, it follows that
\begin{equation*}
\EE_{h\in 1+E} (\bm{1}_{p_{11}(h^{-1}Mh) = 0})
\le q^{-k}.
\end{equation*}
On the other hand, if $p_{21}(M)=0$, then 
\begin{equation*}
\EE_{h\in 1+E} (\bm{1}_{p_{11}(h^{-1}Mh) = 0})
= 1
\end{equation*}
and $(p_{11},p_{21})(h^{-1}Mh) = 0$ for all $h\in 1+E$,
by \eqref{p11} and \eqref{p21}.
Now (1) follows.

(2):
As in the proof of (1),
we may assume that $(p_{11},p_{22})(M) = 0$.
Then
\begin{equation}
\label{p22}
p_{22}(h^{-1}Mh)
= p_{21}(M)p_{12}(N)
\end{equation}
for $h=1+N$,
by \eqref{NMN}.
The map
\begin{equation}
\label{define-map-Phi}
\Phi\maps \FF_q^{k\times (n-k)} \to \mathsf{M}_k(\FF_q)\times \mathsf{M}_{n-k}(\FF_q),
\quad A\mapsto (Ap_{21}(M), p_{21}(M)A)
\end{equation}
is linear in the entries of $A$.
By \eqref{p11} and \eqref{p22},
we have
\begin{equation*}
\EE_{h\in 1+E} (\bm{1}_{(p_{11},p_{22})(h^{-1}Mh) = 0})
= \EE_{N\in E} (\bm{1}_{p_{12}(N)\in \ker\Phi}),
\end{equation*}
If $A\in \ker\Phi$,
then the \emph{rows} of $A$ lie in the left kernel of $p_{21}(M)$,
and the \emph{columns} of $A$ lie in the right kernel of $p_{21}(M)$.
If $p_{21}(M)\ne 0$ and $A\in \ker\Phi$,
then it follows that there exist a column $C$ and row $R$ of $A$
such that the entries of $A$ are
uniquely determined by the entries in $A\setminus (C\cup R)$.
Since $\#(C\cup R) = k+(n-k)-1 = n-1$, it follows that
\begin{equation}
\label{if-p21-non-zero}
\EE_{h\in 1+E} (\bm{1}_{(p_{11},p_{22})(h^{-1}Mh) = 0})
\le q^{-(n-1)}
\end{equation}
if $p_{21}(M)\ne 0$.
On the other hand, if $p_{21}(M)=0$, then
\begin{equation*}
\EE_{h\in 1+E} (\bm{1}_{(p_{11},p_{22})(h^{-1}Mh) = 0})
= 1
\end{equation*}
and $(p_{11},p_{21},p_{22})(h^{-1}Mh) = 0$ for all $h\in 1+E$,
by \eqref{p11}, \eqref{p21}, and \eqref{p22}.
Hence (2) follows.

(3):
As in the proof of (1),
we may assume that $(p_{11},p_{12},p_{22})(M) = 0$.
Then
\begin{equation}
\label{p12}
p_{12}(h^{-1}Mh)
= -p_{12}(N)p_{21}(M)p_{12}(N)
\end{equation}
for $h=1+N$,
by \eqref{NMN}.
By \eqref{p11}, \eqref{p12}, and \eqref{p22},
we have
\begin{equation*}
\EE_{h\in 1+E} (\bm{1}_{(p_{11},p_{12},p_{22})(h^{-1}Mh) = 0})
= \EE_{N\in E} (\bm{1}_{p_{12}(N)\in \ker\Phi}),
\end{equation*}
where $\Phi$ is defined as in \eqref{define-map-Phi}.
However, since $M\ne 0$, we must have $p_{21}(M)\ne 0$.
Therefore, (3) follows from \eqref{if-p21-non-zero}.
\end{proof}

Henceforth, assume $p\ge n$,
so that the following simple combinatorial lemma holds.

\begin{lemma}
\label{k-comb}
Let $1\le k\le n-1$.
Suppose $x\in \FF_q^n$ is non-zero.
Then $\sum_{i\in I} x_i\ne 0$
for some $k$-element subset $I\belongs \{1,2,\dots,n\}$.
\end{lemma}

\begin{proof}
If $x_{n-1}\ne x_n$, say,
then either $x_1+\dots+x_{k-1}+x_{n-1}$ or $x_1+\dots+x_{k-1}+x_n$
must be non-zero.
By symmetry, it remains to consider the case $x_1=\dots=x_n$.
Then $x_1\ne 0$, since $x\ne 0$.
Thus $x_1+\dots+x_k = kx_1\ne 0$, because $p\ge n>k$.
\end{proof}

\begin{lemma}
\label{aux-lem}
Let $M\in \mathsf{M}_n(\FF_q)$, with 
$M\ne 0$.
Let $1\le k\le n-1$.
Then the following hold:
\begin{enumerate}
\item $\EE_{g\in \GL_n(\FF_q)}(\bm{1}_{(p_{11},p_{21})(g^{-1}Mg) = 0})
\ll_n q^{-k}$.

\item $\EE_{g\in \GL_n(\FF_q)}(\bm{1}_{p_{11}(g^{-1}Mg) = 0})
\ll_n q^{-k}$.

\item $\EE_{g\in \GL_n(\FF_q)}(\bm{1}_{(p_{11},p_{12},p_{22})(g^{-1}Mg) = 0})
\ll_n q^{-(n-1)}$.

\item $\EE_{g\in \GL_n(\FF_q)}(\bm{1}_{(p_{11},p_{21},p_{22})(g^{-1}Mg) = 0})
\ll_n q^{-(n-1)}$.

\item $\EE_{g\in \GL_n(\FF_q)}(\bm{1}_{(p_{11},p_{22})(g^{-1}Mg) = 0})
\ll_n q^{-(n-1)}$.

\item $\EE_{g\in \GL_n(\FF_q)}(\bm{1}_{p_{22}(g^{-1}Mg) = 0})
\ll_n q^{-(n-k)}$.

\item $\EE_{g\in \GL_n(\FF_q)}(\bm{1}_{\tr(p_{11}(g^{-1}Mg)) = 0})
\ll_n q^{-1}$.
\end{enumerate}
\end{lemma}

\begin{proof}
Although (2) is stronger than (1),
and (5) is stronger than both (3) and (4),
it turns out that proving (1) first will help in proving (2),
and likewise for (3) and (4).

(1):
Let $e_1,\dots,e_n\in \FF_q^n$ be the usual coordinate basis vectors.
The condition $(p_{11},p_{21})(M)=0$ holds
if and only if $Me_j = 0$ for all $1\le j\le k$.
Thus
\begin{equation*}
\EE_{g\in \GL_n(\FF_q)}(\bm{1}_{(p_{11},p_{21})(g^{-1}Mg)=0})
= \EE_{g\in \GL_n(\FF_q)}(\bm{1}_{g^{-1}Mge_1=\dots=g^{-1}Mge_k=0}).
\end{equation*}
Note that $(ge_1,\dots,ge_k)$ is equally likely to be any list of $k$ linearly independent vectors.
Since the \emph{right} kernel of $M$ has dimension at most $n-1$,
and the number of lists of $k$ linearly independent vectors in $\FF_q^d$
is $\mathscr{V}(d) = \prod_{1\le i\le k} (q^d - q^{i-1})
\le (q^d)^k$,
it follows that
\begin{equation*}
\EE_{g\in \GL_n(\FF_q)}(\bm{1}_{(p_{11},p_{21})(g^{-1}Mg)=0})
\le \frac{\mathscr{V}(n-1)}{\mathscr{V}(n)}
\le \prod_{1\le i\le k} \frac{q^{n-1}}{q^n-q^{i-1}}
\ll_n q^{-k}.
\end{equation*}

(2):
We have
\begin{equation}
\label{E-averaging}
\EE_{g\in \GL_n(\FF_q)}(\bm{1}_{p_{11}(g^{-1}Mg) = 0})
= \EE_{g\in \GL_n(\FF_q)}
\EE_{h\in 1+E}(\bm{1}_{p_{11}((gh)^{-1}Mgh) = 0}).
\end{equation}
Applying Lemma~\ref{E-avg-lem}(1) with $g^{-1}Mg$ in place of $M$, we get
\begin{equation*}
\begin{split}
\EE_{g\in \GL_n(\FF_q)}(\bm{1}_{p_{11}(g^{-1}Mg) = 0})
&\le q^{-k}
+ \EE_{g\in \GL_n(\FF_q)}(\bm{1}_{(p_{11},p_{21})(g^{-1}Mg)=0}) \\
&\ll_n q^{-k},
\end{split}
\end{equation*}
by (1).

(3):
Mimicking \eqref{E-averaging}, we have
\begin{equation*}
\EE_{g\in \GL_n(\FF_q)}(\bm{1}_{(p_{11},p_{12},p_{22})(g^{-1}Mg) = 0})
= \EE_{g\in \GL_n(\FF_q)}
\EE_{h\in 1+E}(\bm{1}_{(p_{11},p_{12},p_{22})((gh)^{-1}Mgh) = 0}).
\end{equation*}
Applying Lemma~\ref{E-avg-lem}(3) with $g^{-1}Mg$ in place of $M$, we get (3).

(4):
Immediate from (3)
with $(M^t,g^t)$ in place of $(M,g)$.

(5):
Mimicking \eqref{E-averaging}, we have
\begin{equation*}
\EE_{g\in \GL_n(\FF_q)}(\bm{1}_{(p_{11},p_{22})(g^{-1}Mg) = 0})
= \EE_{g\in \GL_n(\FF_q)}
\EE_{h\in 1+E}(\bm{1}_{(p_{11},p_{22})((gh)^{-1}Mgh) = 0}).
\end{equation*}
Applying Lemma~\ref{E-avg-lem}(2) with $g^{-1}Mg$ in place of $M$, we get
\begin{equation*}
\begin{split}
\EE_{g\in \GL_n(\FF_q)}(\bm{1}_{(p_{11},p_{22})(g^{-1}Mg) = 0})
&\le q^{-(n-1)}
+ \EE_{g\in \GL_n(\FF_q)}(\bm{1}_{(p_{11},p_{21},p_{22})(g^{-1}Mg)=0}) \\
&\ll_n q^{-(n-1)},
\end{split}
\end{equation*}
by (4).

(6):
This follows from (2)
after replacing $k$ with $n-k$,
and conjugating $M$ by a suitable permutation matrix.

(7):
Assume $q\gg_n 1$.
By the $k=1$ case of (2),
we may conjugate $M$ to assume that its top left entry is non-zero.
Then by Lemma~\ref{k-comb},
there exist $k$ diagonal entries of $M$
whose total is non-zero.
By a further conjugation, we may assume $\tr(p_{11}(M)) \ne 0$.
Therefore, the subvariety $\tr(p_{11}(g^{-1}Mg)) = 0$
of $\GL_n(\ol{\FF}_q)$ is either empty or of codimension one,
since $\GL_n(\ol{\FF}_q)$ is irreducible.
The Lang--Weil bound now implies (7).
\end{proof}

We are finally ready to bound
the quantity $L(V,M)$ from \eqref{define-LVM}.

\begin{lemma}
\label{main-lem}
Let $V,M\in \mathsf{M}_n(\FF_q)$.
Let $f_V(t)\in \FF_q[t]$ be the \emph{radical}
of the characteristic polynomial of $V$.
Then
\begin{equation}
\label{main-lem-goal}
L(V,M)
\ll_n q^{1-\deg f_V} \bm{1}_{\tr(M)=0}
+ \bm{1}_{M=0}.
\end{equation}
\end{lemma}

\begin{proof}
We use strong induction on $n\ge 1$.
Since $1\in C(gVg^{-1})$ for all $g$,
we may assume $\tr(M)=0$,
or else $L(V,M) = 0$.
Moreover, if $\deg f_V=1$ or $M=0$, then the trivial bound $L(V,M)\le 1$ suffices.
Therefore, we may assume from now on that
$\tr(M)=0$,
$M\ne 0$,
and $\deg f_V\ge 2$.

By definition, a \emph{constructible set} is a finite union of locally closed sets.
By the inclusion-exclusion principle, the usual Lang--Weil estimate (for varieties) extends to constructible sets.
By the Lang--Weil estimate
applied to the constructible set
\begin{align*}
&\{g\in \GL_n(\ol{\FF}_q): C(gVg^{-1})\belongs M^\perp\} \\
&= \bigcup_{0\le d\le n^2}
\{g\in \GL_n(\ol{\FF}_q): \dim(C(gVg^{-1})) = \dim(C(gVg^{-1})\cap M^\perp) = d\},
\end{align*}
it suffices to prove the desired inequality \eqref{main-lem-goal}
under the assumption that
$f_V$ splits completely in $\FF_q$.
By conjugation, we may assume that $V$ is in
rational canonical form, which coincides with Jordan normal form
because $f_V$ splits completely.
Since $\deg f_V\ge 2$, the matrix $V$ has at least $2$ distinct eigenvalues.
After permuting Jordan blocks if necessary, we may assume that
$$
V = \inmat{V_1 & 0 \\ 0 & V_2},
$$
where $V_1\in \mathsf{M}_k(\FF_q)$ and $V_2\in \mathsf{M}_{n-k}(\FF_q)$,
with $1\le k\le n-1$,
such that $V_1$ and $V_2$ share no eigenvalues.
Then, in particular,
\begin{equation}
\label{centraliser-block-decomposes}
C(V) = \inmat{C(V_1) & 0 \\ 0 & C(V_2)}.
\end{equation}

Let $p_1\defeq p_{11}$ and $p_2\defeq p_{22}$.
Let
\begin{equation}
\label{define-block-group}
H\defeq \inmat{\GL_k(\FF_q) & 0 \\ 0 & \GL_{n-k}(\FF_q)}
\subseteq \GL_n(\FF_q).
\end{equation}
Since $C(ghVh^{-1}g^{-1}) = g C(hVh^{-1}) g^{-1}$,
and $g^{-1}M^\perp g = (g^{-1}Mg)^\perp$ by the conjugation-invariance of $\tr(\cdot)$,
we have
\begin{equation}
\begin{split}
\label{key-induction-identity}
L(V,M)
&= \EE_{g\in \GL_n(\FF_q)}
\EE_{h\in H} (\bm{1}_{C(ghV(gh)^{-1})\belongs M^\perp}) \\
&= \EE_{g\in \GL_n(\FF_q)}
\EE_{h\in H} (\bm{1}_{C(hVh^{-1})\belongs (g^{-1}Mg)^\perp}) \\
&= \EE_{g\in \GL_n(\FF_q)}
L(V_1,p_1(g^{-1}Mg)) L(V_2,p_2(g^{-1}Mg)),
\end{split}
\end{equation}
where the last step uses
\eqref{centraliser-block-decomposes},
\eqref{define-block-group},
and the block matrix identity
$$
\tr(\inmat{A & 0 \\ 0 & B} M)
= \tr(A p_1(M)) + \tr(B p_2(M))
$$
for $(A,B)\in \mathsf{M}_k(\FF_q)\times \mathsf{M}_{n-k}(\FF_q)$.
By the inductive hypothesis applied to
the two factors of $L$ in \eqref{key-induction-identity}, we get
\begin{equation*}
\begin{split}
L(V,M)
&\ll_n \EE_{g\in \GL_n(\FF_q)}
\prod_{1\le i\le 2} (q^{1-\deg f_{V_i}}
\bm{1}_{\tr(p_i(g^{-1}Mg)) = 0}
+ \bm{1}_{p_i(g^{-1}Mg)=0}) \\
&\le q^{2-\deg f_V} \mathcal{P}_0
+ q^{1-\deg f_{V_1}} \mathcal{P}_1
+ q^{1-\deg f_{V_2}} \mathcal{P}_2
+ \mathcal{P}_3,
\end{split}
\end{equation*}
where
\begin{equation*}
\begin{split}
\mathcal{P}_0 &\defeq \EE_{g\in \GL_n(\FF_q)}(\bm{1}_{\tr(p_1(g^{-1}Mg)) = 0}
\bm{1}_{\tr(p_2(g^{-1}Mg)) = 0}), \\
\mathcal{P}_1 &\defeq \EE_{g\in \GL_n(\FF_q)}(\bm{1}_{\tr(p_1(g^{-1}Mg)) = 0}
\bm{1}_{p_2(g^{-1}Mg) = 0}), \\
\mathcal{P}_2 &\defeq \EE_{g\in \GL_n(\FF_q)}(\bm{1}_{\tr(p_2(g^{-1}Mg)) = 0}
\bm{1}_{p_1(g^{-1}Mg) = 0}), \\
\mathcal{P}_3 &\defeq \EE_{g\in \GL_n(\FF_q)}(\bm{1}_{p_1(g^{-1}Mg) = 0}
\bm{1}_{p_2(g^{-1}Mg) = 0}). \\
\end{split}
\end{equation*}

We now wish to bound the probabilities $\mathcal{P}_i$.
Since $$0 = \tr(M) = \tr(g^{-1}Mg)
= \tr(p_1(g^{-1}Mg)) + \tr(p_2(g^{-1}Mg)),$$
the probabilities $\mathcal{P}_i$ for $0\le i\le 2$ simplify as follows:
\begin{equation*}
\begin{split}
\mathcal{P}_0 &= \EE_{g\in \GL_n(\FF_q)}(\bm{1}_{\tr(p_1(g^{-1}Mg)) = 0}), \\
\mathcal{P}_1 &= \EE_{g\in \GL_n(\FF_q)}(\bm{1}_{p_2(g^{-1}Mg) = 0}), \\
\mathcal{P}_2 &= \EE_{g\in \GL_n(\FF_q)}(\bm{1}_{p_1(g^{-1}Mg) = 0}).
\end{split}
\end{equation*}
By parts~(7),
(6),
(2),
and~(5)
of Lemma~\ref{aux-lem},
respectively, we have
$\mathcal{P}_0 \ll_n q^{-1}$,
$\mathcal{P}_1 \ll_n q^{-(n-k)}
\le q^{-\deg f_{V_2}}$,
$\mathcal{P}_2 \ll_n q^{-k}
\le q^{-\deg f_{V_1}}$,
and $\mathcal{P}_3 \ll_n q^{1-n}
\le q^{1-\deg f_V}$.
Thus
\begin{equation*}
\begin{split}
L(V,M)
&\ll_n q^{1-\deg f_V}
+ q^{1-\deg f_{V_1}} q^{-\deg f_{V_2}}
+ q^{1-\deg f_{V_2}} q^{-\deg f_{V_1}}
+ q^{1-\deg f_V} \\
&= 4q^{1-\deg f_V},
\end{split}
\end{equation*}
since $f_{V_1}f_{V_2} = f_V$.
\end{proof}

Plugging Lemma~\ref{main-lem} into \eqref{orbit-collapse},
we get
\begin{equation*}
\Sigma(M)
\ll_n q^{n^2} \sum_{V\in K_n(\FF_q)}
(q^{1-\deg f_V} \bm{1}_{\tr(M)=0}
+ \bm{1}_{M=0}).
\end{equation*}
But  $K_n(\FF_q)$ is in bijection with
the set of rational canonical forms on $\mathsf{M}_n(\FF_q)$, where we recall 
that the \emph{rational canonical form} of  $M\in \mathsf{M}_n(\FF_q)$ 
consists of a partition $\lambda_\phi
=(\lambda_{\phi,1}\ge \lambda_{\phi,2}\ge \cdots\ge0)$
for each monic irreducible polynomial $\phi\in \FF_q[t]$,
such that the action of $M$ on $\FF_q^n$
is isomorphic to multiplication by $t$ on the vector space
$\bigoplus_\phi
\bigoplus_i (\FF_q[t]/\phi^{\lambda_{\phi,i}}\FF_q[t])$,
where $\sum_\phi \deg(\phi) \abs{\lambda_\phi} = n$.
Thus, by the prime number theorem in $\FF_q[t]$
and the fact that $n$ has only finitely many partitions, we deduce that 
\begin{equation*}
\#\{V\in K_n(\FF_q): \deg f_V = d\}
\ll_n q^d
\end{equation*}
for $1\le d\le n$.
Summing over $1\le d\le n$, we get
$$
\Sigma(M)
\ll_n q^{n^2+1} \bm{1}_{\tr(M)=0}
+ q^{n^2+n} \bm{1}_{M=0}.
$$
Finally, we conclude 
from \eqref{eq:games} that  Theorem~\ref{THM:dimension} holds.


\begin{thebibliography}{999} 


 \bibitem{bas} R. Basili, On the irreducibility of varieties of commuting matrices.
{\em  Pure and Applied Algebra} {\bf 149} (2000), 107--120.

\bibitem{uniformFK} D. Bonolis, E. Kowalski and K. Woo,
Stratification theorems for exponential sums in families.
{\em Preprint}, 2025.
({\tt arXiv:2506.18299})

\bibitem{BPW} D. Bonolis, L. B. Pierce and K. Woo,
Counting integral points in thin sets of type II: singularities, sieves, and stratification.
{\em Preprint}, 2025.
({\tt arXiv:2505.11226})

\bibitem{bhb} T. D.  Browning and D. R. Heath-Brown, 
Rational points on quartic hypersurfaces. {\em J.\ reine angew.\ Math.} {\bf 629} (2009), 37--88.


\bibitem{DGCnew}
W. Castryck, R. Cluckers, P. Dittmann and  K. H. Nguyen,
The dimension growth conjecture, polynomial in the degree and without logarithmic factors.
{\em Algebra \& Number Theory} 
{\bf 14} (2020), 2261--2294.

\bibitem{FeitFine}
W. Feit and  N. J. Fine,
Pairs of commuting matrices over a finite field.
{\em Duke Math.\ J.} {\bf 27} (1960), 91--94.

\bibitem{FK}
E. Fouvry and N. Katz, 
A general stratification theorem for exponential sums, and applications.
{\em J.\ reine angew.\ Math.} {\bf 540} (2001), 115--166.

\bibitem{Fujiwara1}
M. Fujiwara,
Upper bounds for the number of lattice points on hypersurfaces.
{\em Number theory and combinatorics (Japan, 1984)},
89--96,
World Scientific, 1985.

\bibitem{Fujiwara2}
M. Fujiwara,
Distribution of rational points on varieties over finite fields.
{\em Mathematika} {\bf 35} (1988), 155--171.

\bibitem{hart}
R. Hartshorne,
\emph{Algebraic Geometry}.
Grad. Texts in Math. {\bf 52},
Springer-Verlag, 1977.

\bibitem{Heath}
D. R. Heath-Brown,
The density of rational points on non-singular hypersurfaces.
{\em Proc.\ Indian Acad.\ Sci. (Math.\ Sci.)} {\bf 104} (1994), 13--29.

\bibitem{Knutson}
A. Knutson,
Some schemes related to the commuting variety.
{\em J.\ Algebraic Geom.} \textbf{14} (2005), 283--294.

\bibitem{knut} 
A. Knutson and P. Zinn-Justin, A scheme related to the Brauer loop model.
{\em Advances in Math.} {\bf 214} (2007), 40--77.

\bibitem{LarsenLu}
M. Larsen and Z. Lu, 
Flatness of the commutator map over $\SL_n$.
\emph{IMRN} (2021), 5605--5622.

\bibitem{mats}
H. Matsumura,
\emph{Commutative Ring Theory}.
Cambridge University Press,
Cambridge, 1986.

\bibitem{mot} T. S. Motzkin and O. Taussky, Pairs of matrices with property $L$,~II. {\em Trans.\ Amer.\ Math.\ Soc.} {\bf 80} (1955), 387--401.


\bibitem{Neubauer}
M. G. Neubauer,
The variety of pairs of matrices with $\operatorname{rank}(AB - BA) \leq 1$.
\emph{Proc.\ Amer.\ Math.\ Soc.} {\bf 105}
(1989), 787--792.

 \bibitem{Poonen} 
 B. Poonen,
 \emph{Rational Points on Varieties}.
 Graduate Studies in Math.\ {\bf 186},
 Amer.\ Math.\ Soc., Providence, RI, 2017.

\bibitem{salb}
P. Salberger,
Counting rational points on projective varieties.
{\em Proc.\ Lond.\ Math. Soc.} {\bf 126} (2023), 1092--1133.

\bibitem{shoda}
K. Shoda, Einige S\"atze \"uber Matrizen. {\em Japan J.\ Math.} {\bf 13} (1937), 361--365.

\bibitem{SS}
I. E. Shparlinski and A. N. Skorobogatov,
Exponential sums and rational points on complete intersections.
{\em Mathematika} {\bf 37} (1990), 201--208.

\bibitem{Skoro}
A. N. Skorobogatov,
Exponential sums, the geometry of hyperplane sections, and some Diophantine problems.
{\em Israel J.\ Math.} {\bf 80} (1992), 359--379.

\bibitem{stacks-project}
The Stacks project authors,
\emph{The Stacks project}.
\url{https://stacks.math.columbia.edu},
2024.

\bibitem{AffDGC}
F. Vermeulen,
Dimension growth for affine varieties.
\emph{IMRN} (2024), 11464--11483.

\bibitem{Young}
H.-W. V. Young,
On matrix pairs with diagonal commutators.
{\em J.\ Algebra} {\bf 570} (2021), 437--451.

\end{thebibliography}
\end{document}